\documentclass[11pt,oneside]{article}

\usepackage[T1]{fontenc}
\usepackage[utf8]{inputenc}
\usepackage{amsmath,amssymb,amsthm,mathtools}
\usepackage{enumitem}
\usepackage[colorlinks=true,linkcolor=blue,citecolor=blue,urlcolor=blue]{hyperref}

\newtheorem{theorem}{Theorem}[section]
\newtheorem{proposition}[theorem]{Proposition}
\newtheorem{lemma}[theorem]{Lemma}
\newtheorem{corollary}[theorem]{Corollary}
\newtheorem{definition}[theorem]{Definition}

\newcommand{\R}{\mathbb R}
\newcommand{\dd}{\,d}
\newcommand{\DN}{\Lambda}
\newcommand{\Sym}{\operatorname{Sym}}
\newcommand{\diver}{\operatorname{div}}
\newcommand{\supp}{\operatorname{supp}}
\newcommand{\Id}{\operatorname{Id}}
\newcommand{\eps}{\varepsilon}

\title{The Calder\'on problem for piecewise polynomial anisotropic conductivities on partitions with many flat faces}
\author{C\u{a}t\u{a}lin I. C\^{a}rstea\thanks{
Department of Applied Mathematics,
National Yang Ming Chiao Tung University,
Hsinchu 300, Taiwan, R.O.C.,
\texttt{catalin.carstea@gmail.com}}}
\date{}

\begin{document}
\maketitle

\begin{abstract}
We prove uniqueness for a finite-dimensional anisotropic Calder\'on problem in
dimensions \(n\ge 3\).  The unknown conductivity is a symmetric uniformly
elliptic matrix field whose restriction to each cell of a known finite
partition is polynomial; jumps across cell interfaces are allowed.  We assume
that the cells can be ordered so that each successive cell is accessible
through sufficiently many flat faces.  Under this assumption, the local
Dirichlet-to-Neumann map on an exterior boundary patch determines every
polynomial piece.  The proof first recovers one polynomial matrix from its
boundary invariants on several flat faces and then propagates the required
local data through the recovered cells by unique continuation and Runge
approximation.  A general finite-dimensional analytic stability theorem also
gives H\"older stability on compact admissible parameter sets.
\end{abstract}


\section{Introduction}

Let \(\Omega\subset \R^n\), \(n\ge 3\), be a bounded domain and let
\(\sigma(x)\) be a symmetric positive definite matrix field.  We study the
inverse problem of determining \(\sigma\) in the conductivity equation
\begin{equation}
  \diver(\sigma \nabla u)=0\quad\text{in }\Omega,
  \label{eq:conductivity-equation}
\end{equation}
from a local Dirichlet-to-Neumann (DN) map.

For anisotropic conductivities, boundary-fixing diffeomorphisms give the
fundamental obstruction to uniqueness.  If
\(F:\overline\Omega\to\overline\Omega\) is a diffeomorphism satisfying
\(F|_{\partial\Omega}=\Id\), then the push-forward conductivity
\begin{equation}
  (F_*\sigma)(y)=
  \frac{DF(x)\sigma(x)DF(x)^T}{|\det DF(x)|},
  \qquad x=F^{-1}(y),
  \label{eq:gauge}
\end{equation}
has the same full boundary measurements as \(\sigma\); see, for example,
\cite{LassasUhlmann2001,LeeUhlmann1989,Sylvester1990}.  Our aim is to remove
this ambiguity by restricting the conductivity to a finite-dimensional class.

Polynomial dependence alone does not remove the gauge.  To see this, let
\(\Omega=B(0,1)\subset \R^n\) and let \(A\) be a nonzero skew-symmetric
matrix, for example the generator of rotations in the \((x_1,x_2)\)-plane.
Set \(R_\theta=\exp(\theta A)\) and
\begin{equation}
  F_\eps(x)=R_{\eps(1-|x|^2)}x.
  \label{eq:polynomial-gauge-example-map}
\end{equation}
The map \(F_\eps\) is a diffeomorphism of the ball onto itself and is the
identity on \(\partial B(0,1)\).  Moreover, \(Ax\cdot x=0\), so its Jacobian
determinant is one.  The push-forward of the constant conductivity \(I\) is
\begin{equation}
  (F_{\eps*}I)(y)
  =B_\eps(y)B_\eps(y)^T,\qquad
  B_\eps(y)=I-2\eps (Ay)\otimes y.
  \label{eq:polynomial-gauge-example-conductivity}
\end{equation}
For small nonzero \(\eps\), the conductivities \(I\) and \(F_{\eps*}I\) are
distinct, uniformly elliptic, and polynomial, yet their boundary measurements
coincide.

We therefore impose both a known finite partition and polynomial dependence on
each cell.  If \(D_\alpha\) is a cell, then
\begin{equation}
  \sigma|_{D_\alpha}=P_\alpha|_{D_\alpha},
  \qquad
  P_\alpha\in \operatorname{Poly}_{\le N_\alpha}
       (\R^n;\Sym_n).
  \label{eq:piecewise-polynomial}
\end{equation}
The partition and degree bounds are known, whereas the coefficients of the
matrix polynomials \(P_\alpha\) are not.  The pieces need not agree across cell
interfaces.  We ask whether all of them are determined by the DN map on one
relatively open subset of \(\partial\Omega\).

Calder\'on formulated the scalar problem in \cite{Calderon1980}.  For smooth
scalar conductivities in dimensions \(n\ge3\), uniqueness and reconstruction
were established in \cite{SylvesterUhlmann1987,Nachman1988}.  The
two-dimensional problem for nonsmooth and bounded measurable conductivities is
treated in \cite{BrownUhlmann1997,AstalaPaivarinta2006}; see
\cite{Uhlmann2009} for a survey and its connection with electrical impedance
tomography.

Discontinuous conductivities with finitely many parameters have also been
studied extensively.  Analytic structural assumptions and
uniqueness for discontinuous conductivities were studied in
\cite{KohnVogelius1984,KohnVogelius1985} and \cite{Isakov1988}, while Lipschitz stability
for isotropic piecewise constant conductivities on known partitions was proved
in \cite{AlessandriniVessella2005}.  Results for conformally anisotropic
classes, complex admittivities, scalar piecewise linear conductivities, and
fully anisotropic piecewise constant tensors on known partitions include
\cite{BerettaFrancini2011,GaburroSincich2015} and
\cite{AlessandriniDeHoopGaburroSincich2016,AlessandriniDeHoopGaburro2017}.
Related finite-dimensional results for piecewise constant Schr\"odinger
potentials, Helmholtz coefficients, and Lam\'e parameters include
\cite{BerettaDeHoopQiu2012}, \cite{BerettaDeHoopQiuScherzer2014}, and
\cite{BerettaFranciniVessella2013};
see \cite{Garde2020} for layered piecewise constant conductivities.  For
piecewise constant coefficients in elasticity, see
\cite{CarsteaHondaNakamura2018} and \cite{CarsteaNakamuraOksanen2019}.
The inner-extension step below is closest to the known-region continuation
argument of \cite{Ikehata2002}; it was also used in the two elasticity papers
just cited.

The only boundary determination result needed here is the principal-symbol
calculation for the weak anisotropic conductivity DN form.  Relative to the
fixed Euclidean surface density on a flat face, that symbol determines the
conormal Schur form introduced below; we use the local result of
\cite{KangYun2003}.

The first part of the proof is local.  On a flat face with conormal \(\nu\),
boundary determination gives a scaled Schur-complement invariant of \(\sigma\),
which we call the conormal Schur form.  At a point, the Schur forms associated
with three independent conormals determine a positive definite matrix.  For a
matrix polynomial of degree at most \(N\), data on \(N+3\) suitably generic
flat faces then determine the whole polynomial by a divisibility argument.  We
expect that \(\max(3,N+2)\) faces suffice, but use \(N+3\) here because the
resulting argument is considerably simpler.

The second part is a layer-stripping argument.  Once several cells have been
recovered, their union is a known region between the measured boundary and the
remaining cells.  Unique continuation and Runge approximation through this
region determine the local DN map on the newly accessible internal interface.
The main text states this inner-extension result and sketches its proof; the
trace-space and approximation details are given in the appendix.

Accordingly, each recovery step requires enough accessible flat faces, a
genericity condition on their supporting hyperplanes and triple intersections,
and a face-connected chain through the recovered cells to the measured boundary
patch.  These assumptions are adapted to polyhedral and mesh-type models, where
the partition and its flat faces are known.

We next make the geometric assumptions precise and state the main theorem.  Its
proof is given in Section \ref{sec:global-uniqueness}.

\subsection{Assumptions and notation}
\label{sec:local-dn}

Let \(U\subset\R^n\) be a bounded Lipschitz domain and let
\(\Gamma\subset\partial U\) be relatively open.  Remove from \(\Gamma\) the
lower-dimensional edges of the flat subdivision, and denote the resulting
relative interior by \(\Gamma^\circ\).  Throughout the paper we set
\begin{equation}
  H^{1/2}_0(\Gamma;\partial U)
  =\overline{C_c^\infty(\Gamma^\circ)}^{\,H^{1/2}(\partial U)}
\end{equation}
and let \(H^{-1/2}(\Gamma;\partial U)\) be its dual under the usual
\(H^{-1/2}\)-\(H^{1/2}\) pairing.  We suppress \(\partial U\) when the ambient
  boundary is clear.  Thus the initial traces are smooth and compactly supported
  away from the edges, and the full spaces are obtained by completion.  For an
  internal face, Proposition \ref{prop:inner-extension} is proved first for such
  smooth generators on one selected patch and then extended in this ordinary
  local trace norm; no completion over the full multi-face interface is used.

If \(\sigma\in L^\infty(U;\Sym_n)\) is uniformly elliptic, the local
Dirichlet-to-Neumann map on \(\Gamma\) is the bounded operator
\begin{equation}
  \DN^U_{\sigma,\Gamma}:H^{1/2}_0(\Gamma;\partial U)
     \to H^{-1/2}(\Gamma;\partial U)
\end{equation}
defined by
\begin{equation}
  \langle \DN^U_{\sigma,\Gamma}f,h\rangle
  =\int_U \sigma\nabla u_f\cdot \nabla v_h\,\dd x,
  \label{eq:local-dn}
\end{equation}
where \(u_f\in H^1(U)\) is the solution of
\(\diver(\sigma\nabla u_f)=0\) with boundary trace \(f\), and
\(v_h\in H^1(U)\) is any extension of \(h\).  The integral does not depend on
the choice of extension.

By a Lipschitz domain we mean a connected open set with strong Lipschitz
boundary: near each boundary point, after a rigid motion, the domain lies on
one side of a Lipschitz graph.  This convention applies to \(\Omega\), the
cells \(D_\alpha\), and the recovered and remaining domains introduced below;
boundedness will always be stated separately.

We first specify the allowed partitions.  For distinct cells set
\begin{equation}
  I_{\alpha\beta}
  =\Omega\cap\partial D_\alpha\cap\partial D_\beta .
  \label{eq:clipped-cell-interface}
\end{equation}
We say that the subdivision is \emph{internally flat up to a skeleton} if there
is a closed set \(E_{\mathcal D}\subset\overline\Omega\), contained in a finite
union of Lipschitz images of affine sets of dimension at most \(n-2\), with the
following property.  The set \(E_{\mathcal D}\cap I_{\alpha\beta}\) contains no
nonempty relatively open subset of any nonempty interface
\(I_{\alpha\beta}\).  Moreover, if \(x\) belongs to the relative interior of
\(I_{\alpha\beta}\setminus E_{\mathcal D}\), then there are a ball \(B_x\) and
an affine hyperplane \(H_x\) such that
\begin{equation}
  I_{\alpha\beta}\cap B_x\subset H_x,
\end{equation}
and \(D_\alpha\) and \(D_\beta\) occupy opposite one-sided neighborhoods of
\(H_x\) in \(B_x\).  Selected exterior patches in a recovery order are
separately required to be planar and disjoint from the skeleton.  No condition
is imposed on unused portions of \(\partial\Omega\), which may be curved.

\begin{definition}
\label{def:piecewise-poly-class}
Let \(\mathcal D=\{D_\alpha\}_{\alpha\in A}\) be a finite family of pairwise
disjoint connected Lipschitz domains such that
\begin{equation}
  \overline\Omega=\bigcup_{\alpha\in A}\overline{D_\alpha}.
\end{equation}
Given integers \(N_\alpha\ge0\), we say that \(\sigma\) belongs to
\(\mathcal P(\mathcal D,\{N_\alpha\})\) if there are symmetric matrix
polynomials
\begin{equation}
  P_\alpha\in \operatorname{Poly}_{\le N_\alpha}(\R^n;\Sym_n)
\end{equation}
such that \(\sigma=P_\alpha\) a.e. in \(D_\alpha\), and \(\sigma\) is uniformly
elliptic in \(\Omega\).
If all degree bounds equal the same integer \(N\), we abbreviate this class by
\(\mathcal P(\mathcal D,N)\).
\end{definition}

Only the coefficients of the polynomials \(P_\alpha\) are unknown.  Across an
accessible interface, boundary determination is applied from the side of the
cell being recovered, where the conductivity agrees with one polynomial.

The flat faces must satisfy the following genericity condition.

\begin{definition}
\label{def:N-generic}
Let \(N\ge0\).  A finite family \(\{H_i\}_{i=1}^M\) of distinct affine
hyperplanes in \(\R^n\) is called \(N\)-generic if, for each \(i\), there is a
set
\begin{equation}
  J_i\subset \{1,\ldots,M\}\setminus\{i\},
  \qquad |J_i|=N+2,
\end{equation}
such that:
\begin{enumerate}[label=(\alph*)]
  \item whenever \(j,k\in J_i\), \(j\ne k\), the normals of
  \(H_i,H_j,H_k\) are linearly independent;
  \item inside \(H_i\), the induced hyperplanes
  \(H_i\cap H_j\), \(j\in J_i\), form a simple affine arrangement: pairwise
  intersections have codimension two in \(H_i\), and no pairwise intersection
  is contained in a third induced hyperplane.
\end{enumerate}
For \(N=0\) the condition reads: \(|J_i|=2\), the three normals indexed by
\(i\) and \(J_i\) are linearly independent, and condition (b) reduces to the
codimension-two requirement.
\end{definition}

\subsection{Admissible flat-face recovery orders}
\label{sec:admissible-order}

An admissible order must provide enough generic faces for Theorem
\ref{thm:flat-cell-recovery} and must keep the recovered region
face-connected to the measured boundary, as required by Lemma
\ref{lem:runge-known-shell} and Proposition \ref{prop:inner-extension}.

Let \(A\) be the finite index set of cells and let
\(\alpha_1,\ldots,\alpha_L\) be an ordering of \(A\).  Define the remaining
domain at step \(r\) by
\begin{equation}
  \Omega_r=\operatorname{int}\left(
  \bigcup_{s=r}^L \overline{D_{\alpha_s}}\right),
  \qquad r=1,\ldots,L,
  \label{eq:remaining-domain}
\end{equation}
Thus \(\Omega_1=\Omega\), and \(\Omega_{r+1}\) is obtained by removing
\(D_{\alpha_r}\) from \(\Omega_r\).  For \(r>1\), set
\begin{equation}
  G_r=\Omega\setminus\overline{\Omega_r}.
\end{equation}
so that \(G_r\) is the region recovered before step \(r\).

For a polyhedral mesh, the first cell is recovered from exterior faces in the
measured set.  Thereafter a cell is recoverable when enough of its planar faces
on the boundary of the remaining domain are accessible in generic position.
Such a face may lie on the exterior boundary or on the interface with the
recovered region.  Only relatively open face patches count; lower-dimensional
contacts do not.

\begin{definition}
\label{def:stripping-order}
Let \(\Sigma\subset\partial\Omega\) be the initial measured set.  The ordering
\(\alpha_1,\ldots,\alpha_L\) is called an admissible flat-face recovery order
relative to \(\Sigma\) and the degree bounds \(N_\alpha\) if the following
conditions hold for each \(r\):
\begin{enumerate}[label=(\roman*)]
  \item \(\Omega_r\) is a Lipschitz domain.  If \(r>1\), then the already
  recovered region \(G_r\) is a nonempty face-connected Lipschitz domain whose
  boundary contains a nonempty relatively open flat patch
  \(\Sigma_r\subset\Sigma\).  This patch contains a relatively compact subpatch
  with an auxiliary exterior half-ball: there is a ball \(B_r\) such that, after
  an affine change of coordinates in \(B_r\),
  \begin{equation}
    \Omega\cap B_r=\{x_n>0\}\cap B_r,
    \qquad
    B_r\cap\partial\Omega=\{x_n=0\}\cap B_r\Subset\Sigma_r .
  \end{equation}
  Face-connectedness is understood with respect to the recovered cells: each
  cell in \(G_r\) can be joined to the cell adjacent to \(\Sigma_r\) by a chain
  of recovered cells meeting along nonempty relatively open flat faces.
  \item The cell \(D_{\alpha_r}\) has flat accessible patches
  \begin{equation}
    \Gamma_{r,i}\subset \partial D_{\alpha_r}\cap\partial\Omega_r,
    \qquad i=1,\ldots,M_r,
  \end{equation}
  contained in relative interiors of flat faces and disjoint from
  \(E_{\mathcal D}\).  For \(r=1\), all these
  patches lie in the measured set \(\Sigma\).  For \(r>1\), each patch
  is either contained in \(\Sigma\) or lies in
  \(\Omega\cap\partial G_r\cap\partial\Omega_r\), that is, on the part of the
  interface with the recovered region lying inside the ambient domain.  The
  intersection with \(\Omega\) excludes any common outer-boundary rim.
  \item \(M_r\ge N_{\alpha_r}+3\), and the supporting hyperplanes
  \(H_{r,1},\ldots,H_{r,M_r}\) of the accessible patches are
  \(N_{\alpha_r}\)-generic.  We fix sets \(J_{r,i}\), \(|J_{r,i}|=N_{\alpha_r}+2\),
  witnessing Definition \ref{def:N-generic}.
  \item For each \(i\) and each pair of distinct indices \(j,k\in J_{r,i}\), there
  is a bounded nonempty relatively open set with compact closure
  \begin{equation}
    W_{r,ijk}\Subset H_{r,i}\cap H_{r,j}\cap H_{r,k}.
  \end{equation}
  The local recovery argument compares the three corresponding Schur forms on
  these sets.  When \(n=3\), the triple intersection is zero-dimensional and
  \(W_{r,ijk}\) may be its single point.
  \item For \(r>1\), put
  \begin{equation}
    \Xi_r=\Omega\cap\partial G_r\cap\partial\Omega_r,
  \end{equation}
  the full interface between the recovered and remaining regions.  Every
  accessible patch not contained in \(\Sigma\) lies in \(\Xi_r\), which is a
  finite union of relatively open flat patches up to lower-dimensional edges
  and corners.
\end{enumerate}
\end{definition}

The remaining domains appear explicitly because inner extension is applied to
the connected recovered region, rather than merely to the adjacency graph of
the partition.  Its face-connectedness and piecewise polynomial coefficient
are precisely the hypotheses used for the Runge approximation in Lemma
\ref{lem:runge-known-shell}.

The hypotheses are realized, for example, by a polyhedral shell.  Let
\(D_2\Subset\Omega\) be a polytope and let
\(D_1=\Omega\setminus\overline{D_2}\), where \(\Omega\) is a larger polytope
chosen so that \(D_1\) is connected.  Recover \(D_1\) first from selected outer
faces contained in \(\Sigma\), and then recover \(D_2\) from selected faces in
\begin{equation}
  \Xi_2=\Omega\cap\partial D_1\cap\partial D_2=\partial D_2.
\end{equation}
Choose the supporting hyperplanes in general position and take
\(J_{r,i}\) from the other selected faces.  Degrees \(0,1,2\) require at least
\(3,4,5\) selected faces, respectively.  The sets \(W_{r,ijk}\) may be small
bounded neighborhoods inside the corresponding triple affine intersections
(singletons in dimension three).  This example also illustrates why the
ambient clipping in the definition of \(\Xi_r\) matters when an inner cell
meets the outer boundary: a common outer-boundary rim is not an internal face.

\subsection{Main result}

\begin{theorem}
\label{thm:main}
Let \(n\ge3\), let \(\Omega\subset\R^n\) be a bounded Lipschitz domain, and let
\(\mathcal D=\{D_\alpha\}_{\alpha\in A}\) be a known finite Lipschitz
subdivision of \(\Omega\) that is internally flat up to a skeleton.  Let
\(N_\alpha\ge0\).  Suppose that
\(\sigma^{(1)},\sigma^{(2)}\in\mathcal P(\mathcal D,\{N_\alpha\})\), and assume
that, at every recovery step and for \(j=1,2\), the polynomial extension of
\(\sigma^{(j)}|_{D_{\alpha_r}}\) is positive definite on the triple-intersection
sets \(W_{r,ijk}\) appearing in Definition \ref{def:stripping-order}.  We also
  assume that, for \(j=1,2\), the polynomial extensions in the recovered cells
  adjacent to the auxiliary exterior half-balls remain elliptic in those
  half-balls whenever the half-balls are used.

Let \(\Sigma\subset\partial\Omega\) be a nonempty relatively open measured set.
Assume that the subdivision admits an admissible flat-face recovery order
relative to \(\Sigma\), in the sense of Definition \ref{def:stripping-order}.
If
\begin{equation}
  \DN^\Omega_{\sigma^{(1)},\Sigma}
  =\DN^\Omega_{\sigma^{(2)},\Sigma},
  \label{eq:outer-dn-equality}
\end{equation}
then
\begin{equation}
  \sigma^{(1)}=\sigma^{(2)}\quad\text{in }\Omega.
\end{equation}
Equivalently, the polynomial matrices defining the two conductivities agree on
every cell of the subdivision.

\end{theorem}

The auxiliary-half-ball assumption is harmless for a fixed admissible
coefficient but useful when the geometry is fixed in advance.  Indeed, if a
matrix polynomial is uniformly positive definite in a bounded cell, continuity
extends the same lower bound to the cell closure.  Compactness then gives an
open neighborhood of that closure on which the polynomial remains positive
definite.  An exterior half-ball attached to the cell can therefore be shrunk
to lie in this neighborhood.  We nevertheless retain the explicit half-ball
condition in Theorem \ref{thm:main}, since the same prescribed auxiliary set is
used for both conductivities and, below, uniformly over parameter families.

The polynomial coefficients form the finite-dimensional parameter space
\begin{equation}
  \mathcal V=\prod_{\alpha\in A}
  \operatorname{Poly}_{\le N_\alpha}(\R^n;\Sym_n).
\end{equation}
For \(P=(P_\alpha)_{\alpha\in A}\in\mathcal V\), let \(\sigma_P=P_\alpha\)
in \(D_\alpha\).  Fix the subdivision, recovery order, measured patch, the
compact sets \(\overline{W_{r,ijk}}\), and the closures of the auxiliary
half-balls.  Let \(\mathcal A\subset\mathcal V\) consist of coefficients for
which the relevant polynomial matrices are positive definite on all cell
closures, all \(\overline{W_{r,ijk}}\), and all auxiliary half-ball closures.
This is an open subset of \(\mathcal V\).  Indeed, for any
\(P\in\mathcal A\), the least eigenvalue has a positive minimum on each of
these compact sets.  There are only finitely many such minima, so their minimum
is a common positive margin; sufficiently small changes of the polynomial
coefficients preserve it.  This also makes clear that the assertion concerns a
fixed finite collection of compact sets, not positivity on an unbounded affine
intersection.

\begin{corollary}
\label{cor:holder-stability}
Let \(\mathcal A\subset\mathcal V\) be as above and let
\(K\Subset \mathcal A\) be compact.  Then there are constants
\(C>0\) and \(\theta\in(0,1]\) such that, for all \(P,Q\in K\),
\begin{equation}
  \max_{\alpha\in A}
  \|P_\alpha-Q_\alpha\|_{C(\overline{D_\alpha})}
  \le
  C\bigl\|
  \DN^\Omega_{\sigma_P,\Sigma}
  -\DN^\Omega_{\sigma_Q,\Sigma}
  \bigr\|_{\mathcal L(H^{1/2}_0(\Sigma),H^{-1/2}(\Sigma))}^{\theta}.
  \label{eq:holder-stability-corollary}
\end{equation}
The constants depend on \(K\), the fixed subdivision and recovery order, and the
chosen norms, but not on \(P,Q\in K\).
\end{corollary}

\begin{proof}
Consider the parameter-to-data map
\begin{equation}
  F(P)=\DN^\Omega_{\sigma_P,\Sigma}
\end{equation}
from \(\mathcal A\) to
\(\mathcal L(H^{1/2}_0(\Sigma),H^{-1/2}(\Sigma))\).  This map is real analytic
in the operator norm.  To see this, fix a bounded right inverse for the trace
map.  The zero-boundary part of the solution is obtained by inverting the
uniformly elliptic Dirichlet operator
\begin{equation}
  L_P:H^1_0(\Omega)\to H^{-1}(\Omega),
  \qquad
  \langle L_P u,v\rangle=\int_\Omega \sigma_P\nabla u\cdot\nabla v\,\dd x .
\end{equation}
The dependence of \(L_P\) on \(P\) is affine, and inversion is real analytic on
the open set of bounded isomorphisms.  The weak formula for the DN map therefore
shows that \(F\) is real analytic.

By Theorem \ref{thm:main},
\begin{equation}
  F(P)=F(Q)\quad\Longrightarrow\quad P=Q,
  \qquad P,Q\in\mathcal A .
\end{equation}
Apply \cite[Theorem 3.1]{Carstea2026HolderStability} with recovered quantity
\(R(P)=P\).  It gives a H\"older estimate for \(P-Q\) on \(K\), and norm
equivalence in the finite-dimensional space \(\mathcal V\) yields
\eqref{eq:holder-stability-corollary}.
\end{proof}

Section \ref{sec:flat-cell} proves the one-cell theorem from Schur invariants on
flat faces.  Section \ref{sec:inner-extension} shows how a local DN map can be
continued across a known region to an internal interface.  These two results
are combined in Section \ref{sec:global-uniqueness} to prove Theorem
\ref{thm:main}.

The theorem extends the piecewise constant anisotropic result of
\cite{AlessandriniDeHoopGaburro2017} to polynomial pieces.  In that work,
variation of the tangent plane along a curved boundary supplies enough
directions to identify the anisotropic tensor.  Here the independent conormals
of several flat faces play the same role: their Schur forms together remove the
Euclidean gauge.  Inner extension supplies those invariants at internal faces,
where the one-cell theorem can then be applied.

\section{Flat-cell recovery}
\label{sec:flat-cell}

We isolate the local problem used at each recovery step.  On a flat face, the
zeroth-order boundary invariant is the conormal Schur form defined below.  For
a polynomial conductivity, equality of Schur forms on an open patch extends to
the supporting hyperplane.  Three hyperplanes with independent conormals then
determine the matrix on their common intersection.  A divisibility argument on
the hyperplane arrangement converts this pointwise information into equality
of the matrix polynomials.

\subsection{The Schur form as a polynomial face invariant}

Let \(V\simeq\R^n\), let \(H\subset V\) be an affine hyperplane, and choose a
nonzero conormal \(\ell\in V^*\).  We regard a conductivity matrix \(A\) as a
positive definite symmetric bilinear form on covectors.  Its conormal Schur
form is
\begin{equation}
  S_{\ell}(A)([\xi],[\eta])
  =A(\ell,\ell)A(\xi,\eta)-A(\ell,\xi)A(\ell,\eta),
  \qquad [\xi],[\eta]\in V^*/\R\ell\simeq T^*H .
  \label{eq:conormal-schur-form}
\end{equation}
The expression is independent of the representatives \(\xi\) and \(\eta\).
Rescaling \(\ell\) by \(c\) rescales \(S_\ell(A)\) by \(c^2\).  In particular,
\(S_{-\ell}=S_\ell\), and equality of two Schur forms is independent of the
normalization and orientation of the conormal.  We write \(S_H(A)\) after
fixing a Euclidean unit conormal to \(H\).

If Euclidean coordinates are chosen so that \(H=\{x_n=0\}\) and
\(\ell=dx_n\), and if
\begin{equation}
  A=
  \begin{pmatrix}
    A_T & b\\ b^T & c
  \end{pmatrix},
  \label{eq:block-form}
\end{equation}
then the Schur form is the usual scaled Schur complement,
\begin{equation}
  S_H(A)=cA_T-bb^T,
  \label{eq:schur-form}
\end{equation}
viewed as a quadratic form on \(T^*H\).  The boundary-fixing change of variables
with differential
\begin{equation}
  DF|_H=
  \begin{pmatrix}
    I_{n-1} & -b/c\\ 0 & 1/c
  \end{pmatrix},
\end{equation}
puts the pushed-forward conductivity \(F_*A\), at the face, in a form with
zero normal--tangential block, normal--normal entry one, and tangential block
\(S_H(A)\).  Thus \(S_H(A)\) is the part of the conductivity visible through
the boundary-fixing gauge at that face.

For later algebra it is useful to record the associated metric.  For
\(n\ge3\), set
\begin{equation}
  g_\sigma=(\det \sigma)^{1/(n-2)}\sigma^{-1},
  \label{eq:metric-from-conductivity}
\end{equation}
as a covariant metric.  Then the conductivity equation can be written as the
Laplace--Beltrami equation for \(g_\sigma\), and
\begin{equation}
  \sigma=(\det g_\sigma)^{1/2}\,g_\sigma^{-1}.
  \label{eq:conductivity-from-metric}
\end{equation}
Let \(h_H(\sigma)=g_\sigma|_{TH}\) be the induced tangential metric.  In the
coordinates above,
\begin{equation}
  h_H(\sigma)=(\det S_H(\sigma))^{1/(n-2)}S_H(\sigma)^{-1},
  \qquad
  S_H(\sigma)=(\det h_H(\sigma))h_H(\sigma)^{-1}.
  \label{eq:metric-schur-relation}
\end{equation}
To verify these identities, put \(B=A_T-bc^{-1}b^T\).  The block inverse
formula gives \(h_H(\sigma)=(\det\sigma)^{1/(n-2)}B^{-1}\), while
\(\det\sigma=c\det B\) and \(S_H(\sigma)=cB\).  This proves the first identity;
  the second follows by taking determinants.  The two algebraic maps are inverse on the
positive definite cone: if
\(h=(\det S)^{1/(n-2)}S^{-1}\), then
\(\det h=(\det S)^{1/(n-2)}\) and therefore \(S=(\det h)h^{-1}\).

For polynomial conductivities, the resulting face invariant is again
polynomial.

\begin{lemma}
\label{lem:schur-polynomial}
Let \(P\in\operatorname{Poly}_{\le N}(\R^n;\Sym_n)\), let \(H\subset\R^n\) be
an affine hyperplane with a fixed unit conormal, and fix affine coordinates on
\(H\).  Then every entry of \(x\mapsto S_H(P(x))\), \(x\in H\), is a
polynomial of degree at most \(2N\) on \(H\).  In particular, if
\(P_1,P_2\in\operatorname{Poly}_{\le N}(\R^n;\Sym_n)\) and
\(S_H(P_1)=S_H(P_2)\) on a nonempty relatively open subset of \(H\), then
\(S_H(P_1)=S_H(P_2)\) on all of \(H\).
\end{lemma}

\begin{proof}
Each entry of \(S_H(P(x))\) is a quadratic expression with constant
coefficients in the entries of \(P(x)\), so it has degree at most \(2N\) on
\(H\).  The final assertion follows from the identity theorem for polynomials
on \(\R^{n-1}\).
\end{proof}

Lemma \ref{lem:schur-polynomial} does not require positivity: \(S_H(P)\) is
polynomial on all of \(H\), even if \(P\) ceases to be elliptic away from the
cell.

We now read this invariant directly from the weak conductivity DN form.  The
normalization of the surface density is important here.  The pairing in
\eqref{eq:local-dn} uses the fixed Euclidean surface measure on the flat face.
If instead one writes a Laplace--Beltrami DN operator using its natural metric
surface measure, conversion back to Euclidean measure contributes a density
factor.  Thus the conductivity and Riemannian DN operators should not simply be
identified without specifying their pairings; after the density conversion,
the symbol below is equivalent to \eqref{eq:metric-schur-relation}.

\begin{lemma}
\label{lem:face-data-from-local-dn}
Let \(W\) be a Lipschitz domain whose boundary is flat in a neighborhood of a
relatively open patch \(\Gamma\subset H\), and suppose that \(W\) lies on one
side of \(H\) near \(\Gamma\).  Let \(\sigma^{(1)},\sigma^{(2)}\) be smooth
uniformly elliptic anisotropic conductivities in a one-sided neighborhood of
\(\Gamma\).  If
\begin{equation}
  \DN^W_{\sigma^{(1)},\Gamma}=\DN^W_{\sigma^{(2)},\Gamma},
\end{equation}
then
\begin{equation}
  S_H(\sigma^{(1)})=S_H(\sigma^{(2)})
  \qquad\text{on }\Gamma .
  \label{eq:face-data-from-local-dn}
\end{equation}
\end{lemma}

\begin{proof}
Choose Euclidean coordinates in which \(H=\{x_n=0\}\), with \(W\) locally on
the side \(x_n>0\), and write at a boundary point
\begin{equation}
  \sigma=
  \begin{pmatrix}
    A&b\\ b^T&c
  \end{pmatrix},
  \qquad c>0.
\end{equation}
The frozen half-space equation with tangential covector \(\xi\) has a decaying
solution of the form
\(e^{ix'\cdot\xi-\lambda x_n}\).  Choosing the root with positive real part and
taking the outward conormal gives
\begin{equation}
  c\lambda-i b\cdot\xi
  =\sqrt{\xi^T(cA-bb^T)\xi}.
  \label{eq:weak-dn-principal-symbol}
\end{equation}
Consequently the principal symbol of the weak conductivity DN form, relative
to the fixed Euclidean density, is
\begin{equation}
  \xi\longmapsto\sqrt{\xi^T S_H(\sigma(x))\xi}.
\end{equation}
This local symbol calculation, including the localization from the weak local
DN form, is the boundary determination result of \cite{KangYun2003}.  Equality
of the two local forms therefore gives equality of the displayed quadratic
forms for every tangential \(\xi\).  Squaring and polarizing yields
\(S_H(\sigma^{(1)}(x))=S_H(\sigma^{(2)}(x))\) at every \(x\in\Gamma\).
\end{proof}

\subsection{Three faces determine the matrix pointwise}

The next lemma shows that three independent faces remove the pointwise gauge
left by a single face.

\begin{lemma}
\label{lem:three-schur-pointwise}
Let \(H_1,H_2,H_3\) be affine hyperplanes through a point \(x_0\), and let
\(\ell_r\) be nonzero conormals to \(H_r\) at \(x_0\).  Assume that
\(\ell_1,\ell_2,\ell_3\) are linearly independent.  Let \(A,B\in \Sym_n\) be
positive definite, viewed as bilinear forms on covectors.  If
\begin{equation}
  S_{\ell_r}(A)=S_{\ell_r}(B),\qquad r=1,2,3,
  \label{eq:three-schur-equality}
\end{equation}
as quadratic forms on \(T^*_{x_0}H_r\), then \(A=B\).
\end{lemma}

\begin{proof}
We may normalize the conormals, since rescaling contributes the same factor to
both sides of \eqref{eq:three-schur-equality}.  Set
\(g_A=(\det A)^{1/(n-2)}A^{-1}\) and
\(g_B=(\det B)^{1/(n-2)}B^{-1}\), positive definite bilinear forms on vectors,
and let \(T_r=\ker\ell_r\) be the tangent space of \(H_r\) at \(x_0\).  By
\eqref{eq:metric-schur-relation}, the hypothesis
\eqref{eq:three-schur-equality} is equivalent to
\begin{equation}
  g_A|_{T_r}=g_B|_{T_r},\qquad r=1,2,3.
  \label{eq:three-tangential-metrics}
\end{equation}

Choose linear coordinates
\((x_1,\ldots,x_n)\) in which \(\ell_r=dx_r\) for \(r=1,2,3\); this is possible
because the conormals are linearly independent.  Then
\(T_r=\operatorname{span}\{e_j:j\ne r\}\).  Any pair of indices
\(i,j\in\{1,\ldots,n\}\) omits at least one \(r\in\{1,2,3\}\), and for that
\(r\) we have \(e_i,e_j\in T_r\), so
\(g_A(e_i,e_j)=g_B(e_i,e_j)\) by \eqref{eq:three-tangential-metrics}.  Hence
\(g_A=g_B\), and \eqref{eq:conductivity-from-metric} gives \(A=B\).
\end{proof}

\subsection{Vanishing on a hyperplane arrangement}

For an affine hyperplane \(L\subset E\), a \emph{defining affine function} is a
nonconstant affine function \(\ell:E\to\R\) such that \(L=\{\ell=0\}\); it is
unique up to a nonzero scalar.  We write \(\operatorname{Poly}(L)\) for the
ring of polynomial functions on \(L\), identified with
\(\operatorname{Poly}(E)/(\ell)\).  Divisibility below is understood in this
coordinate ring.

\begin{lemma}
\label{lem:hermite-arrangement}
Let \(E\) be an affine space of dimension at least two, and let
\(L_1,\ldots,L_s\subset E\), \(s\ge2\), be distinct affine hyperplanes with
simple pairwise intersections: \(L_a\cap L_b\) has codimension two in \(E\)
for \(a\ne b\), and no such intersection is contained in a third \(L_c\).
If \(q\in \operatorname{Poly}_{\le s-2}(E)\) vanishes on every pairwise
intersection \(L_a\cap L_b\), then \(q=0\).
\end{lemma}

\begin{proof}
Choose defining affine functions \(\ell_1,\ldots,\ell_s\) for
\(L_1,\ldots,L_s\), and fix \(r\).  For every \(a\ne r\), the intersection
\(L_a\cap L_r\) is a hyperplane inside the affine space \(L_r\), and
\(q|_{L_r}\) vanishes on it.  Since
\(L_a\cap L_r=\{\ell_a|_{L_r}=0\}\), this implies that
\(\ell_a|_{L_r}\) divides \(q|_{L_r}\) in \(\operatorname{Poly}(L_r)\).
This is the factor theorem after choosing affine coordinates on \(L_r\) with
\(\ell_a|_{L_r}\) as one coordinate.

These hyperplanes in \(L_r\) are distinct.  If
\(L_a\cap L_r=L_b\cap L_r\) for some distinct \(a,b\ne r\), then this
codimension-two intersection in \(E\) would be contained in the third
hyperplane \(L_b\), contradicting the simplicity assumption.  Hence the linear
factors \(\ell_a|_{L_r}\), \(a\ne r\), are pairwise coprime in
\(\operatorname{Poly}(L_r)\), and their product divides \(q|_{L_r}\):
\begin{equation}
  \prod_{a\ne r}\ell_a|_{L_r}\;\;\text{divides}\;\;q|_{L_r}.
\end{equation}
The product has degree \(s-1\), whereas \(\deg(q|_{L_r})\le s-2\); therefore
\(q|_{L_r}=0\).

Thus \(q\) vanishes on every \(L_r\).  Applying the factor theorem in \(E\)
shows that each \(\ell_r\) divides \(q\).  The factors are pairwise coprime, so
their degree-\(s\) product divides \(q\).  Since \(\deg q\le s-2\), it follows
that \(q=0\).
\end{proof}

\subsection{The one-cell recovery theorem}

\begin{theorem}
\label{thm:flat-cell-recovery}
Let \(n\ge3\).  Let \(D\subset\R^n\) be a cell and let
\begin{equation}
  \Gamma_i\subset \partial D\cap H_i,
  \qquad i=1,\ldots,M,
\end{equation}
be nonempty relatively open flat patches contained in the relative interiors of
flat faces of \(D\).  Let
\begin{equation}
  P_j\in\operatorname{Poly}_{\le N}(\R^n;\Sym_n),
  \qquad j=1,2,
\end{equation}
be uniformly positive definite in the cell \(D\).  Assume:
\begin{enumerate}[label=(\roman*)]
  \item \(M\ge N+3\) and the hyperplanes \(H_1,\ldots,H_M\) are
  \(N\)-generic; fix sets \(J_i\), \(|J_i|=N+2\), witnessing Definition
  \ref{def:N-generic};
  \item for every \(i\) and every two distinct indices \(j,k\in J_i\), there is
  a bounded nonempty relatively open subset with compact closure
  \begin{equation}
    W_{ijk}\Subset H_i\cap H_j\cap H_k
    \label{eq:common-triple-region}
  \end{equation}
  on which \(P_1\) and \(P_2\) are both positive definite.
\end{enumerate}
If the Schur forms agree on all accessible patches,
\begin{equation}
  S_{H_i}(P_1)=S_{H_i}(P_2)
  \qquad\text{on }\Gamma_i,
  \qquad i=1,\ldots,M,
  \label{eq:flat-cell-face-data-hypothesis}
\end{equation}
then \(P_1=P_2\) as matrix polynomials.  Consequently the two conductivity
pieces agree in \(D\).
\end{theorem}

Condition (a) in Definition \ref{def:N-generic} implies that each triple
intersection
\(H_i\cap H_j\cap H_k\) appearing in (ii) is a nonempty affine subspace of
dimension \(n-3\); when \(n=3\) it is a single point, and \(W_{ijk}\) is that
point.

\begin{proof}
\emph{Step 1: Extend the face data to the supporting hyperplanes.}
By Lemma \ref{lem:schur-polynomial}, the entries of \(S_{H_i}(P_1)\) and
\(S_{H_i}(P_2)\) are polynomials on \(H_i\), so the equality
\eqref{eq:flat-cell-face-data-hypothesis} on the nonempty relatively open
patch \(\Gamma_i\) extends to
\begin{equation}
  S_{H_i}(P_1)=S_{H_i}(P_2)
  \qquad\text{on all of }H_i,
  \qquad i=1,\ldots,M.
  \label{eq:schur-equality-whole-hyperplane}
\end{equation}

\emph{Step 2: Recover the matrix on triple intersections.}
Fix \(i\) and distinct \(j,k\in J_i\), and put
\(F_{ijk}=H_i\cap H_j\cap H_k\).  Let \(x\in W_{ijk}\).  The three hyperplanes
\(H_i,H_j,H_k\) pass through \(x\) with linearly independent conormals, the
matrices \(P_1(x)\) and \(P_2(x)\) are positive definite, and by
\eqref{eq:schur-equality-whole-hyperplane} the three Schur form equalities hold
at \(x\).  Lemma \ref{lem:three-schur-pointwise} gives \(P_1(x)=P_2(x)\).
Hence \(Q=P_1-P_2\) vanishes on the nonempty
relatively open subset of the affine subspace \(F_{ijk}\).  The entries of
\(Q\) are polynomials, so by the identity theorem on \(F_{ijk}\),
\begin{equation}
  Q=0\qquad\text{on }F_{ijk}
  \qquad\text{for every }i\text{ and all distinct }j,k\in J_i.
  \label{eq:vanishing-on-triples}
\end{equation}

\emph{Step 3: Show that the difference vanishes on each supporting hyperplane.}
Fix \(i\) and set
\begin{equation}
  E=H_i,
  \qquad L_j=H_i\cap H_j\subset H_i,
  \qquad j\in J_i .
\end{equation}
The genericity conditions make \(\{L_j:j\in J_i\}\) a family of \(s=N+2\)
distinct hyperplanes of \(E\) with simple pairwise intersections, and
\(L_j\cap L_k=F_{ijk}\).  Every scalar entry
\(q=(P_1-P_2)_{\beta\gamma}|_{H_i}\) is a polynomial of degree at most
\(N=s-2\) on \(E\), and it vanishes on every \(L_j\cap L_k\) by
\eqref{eq:vanishing-on-triples}.  Lemma \ref{lem:hermite-arrangement} therefore
gives \(q=0\), and hence
\begin{equation}
  (P_1-P_2)|_{H_i}=0
  \qquad\text{for every }i=1,\ldots,M .
\end{equation}

\emph{Step 4: Conclude by divisibility.}
Let \(\ell_i\) be a defining affine function for \(H_i\).  For each scalar
entry \(Q_{\beta\gamma}=(P_1-P_2)_{\beta\gamma}\), the identity
\(Q_{\beta\gamma}|_{H_i}=0\) implies that \(\ell_i\) divides
\(Q_{\beta\gamma}\).  The hyperplanes are distinct, hence the factors
\(\ell_i\) are pairwise coprime, and
\begin{equation}
  \ell_1\ell_2\cdots\ell_M \quad\text{divides}\quad Q_{\beta\gamma} .
\end{equation}
Since \(M\ge N+3>N\ge\deg Q_{\beta\gamma}\), this forces
\(Q_{\beta\gamma}=0\).  The same holds for every entry, so \(P_1=P_2\).
\end{proof}

Lemma \ref{lem:face-data-from-local-dn} provides the Schur data required by
Theorem \ref{thm:flat-cell-recovery}.  Namely, suppose that \(W\) has flat
boundary near \(\Gamma_i\), its interior side there lies in \(D\), and
\begin{equation}
  \DN^W_{\sigma^{(1)},\Gamma_i}=\DN^W_{\sigma^{(2)},\Gamma_i},
\end{equation}
then \(S_{H_i}(P_1)=S_{H_i}(P_2)\) on \(\Gamma_i\) by Lemma
\ref{lem:face-data-from-local-dn}.  No orientation convention is needed because
\(S_{-\ell}=S_\ell\).  The same observation applies to exterior faces and to
internal faces reached by inner extension.

\section{Inner extension of the local DN map}
\label{sec:inner-extension}

Suppose that the conductivity is known in a connected region between the
measured boundary and an internal interface.  The next proposition continues
the local DN map across this known region to the interface.  It is the scalar
counterpart of the inner-extension argument used for inclusions in
\cite{Ikehata2002} and for piecewise homogeneous elasticity in
\cite{CarsteaHondaNakamura2018}.

Let \(\Omega_2\subset\Omega_1\subset\R^n\) be bounded Lipschitz domains and put
\begin{equation}
  G=\Omega_1\setminus\overline{\Omega_2},
  \qquad
  \Sigma_1=\partial\Omega_1\cap\partial G,
  \qquad
  \Sigma_2=\Omega_1\cap\partial\Omega_2\cap\partial G .
\end{equation}
Here \(G\) is the known region and \(\Sigma_2\) is its full interface with the
remaining domain.  The proposition concerns one selected flat patch
\(\Gamma_2\subset\Sigma_2\), away from the lower-dimensional skeleton.  Its
proof uses quotient traces of actual functions in \(H^1_0(\Omega_1)\), rather
than a separate function space on the full interface.  We do not assume that
\(\Omega_2\Subset\Omega_1\); all closures and neighborhoods involving
\(\Omega_2\) are relative to \(\Omega_1\).

\begin{proposition}
\label{prop:inner-extension}
Let \(\Omega_2\subset\Omega_1\subset\R^n\) be bounded Lipschitz domains such
that \(G=\Omega_1\setminus\overline{\Omega_2}\) is a nonempty face-connected
Lipschitz known region of the kind described in Lemma
\ref{lem:piecewise-poly-ucp}.  Let \(\Gamma_1\subset\Sigma_1\) be a nonempty
relatively open flat patch containing a relatively compact subpatch with an
auxiliary exterior half-ball: after an affine change of coordinates in a ball
\(B\),
\begin{equation}
  \Omega_1\cap B=\{x_n>0\}\cap B,
  \qquad
  B\cap\partial\Omega_1=\{x_n=0\}\cap B\Subset \Gamma_1,
\end{equation}
and the coefficient in the adjacent cell of \(G\) extends as a uniformly
elliptic polynomial coefficient to \(B\).  Let
\(\Gamma_2\subset\Sigma_2\) be a relatively open patch contained in the
relative interior of one flat face, and hence disjoint from the interface
skeleton.

Let \(\sigma^{(1)},\sigma^{(2)}\) be uniformly elliptic conductivities on
\(\Omega_1\), and assume that
\begin{equation}
  \sigma^{(1)}=\sigma^{(2)}=\sigma^0
  \quad\text{in }G,
\end{equation}
where \(\sigma^0\) is known in \(G\), is piecewise polynomial there, and is
extended to a known uniformly elliptic coefficient on \(\Omega_1\).  If
\begin{equation}
  \DN^{\Omega_1}_{\sigma^{(1)},\Gamma_1}
  =\DN^{\Omega_1}_{\sigma^{(2)},\Gamma_1},
  \label{eq:outer-local-equality-inner-section}
\end{equation}
then the local DN maps on the selected patch agree:
\begin{equation}
  \DN^{\Omega_2}_{\sigma^{(1)},\Gamma_2}
  =\DN^{\Omega_2}_{\sigma^{(2)},\Gamma_2}.
  \label{eq:inner-dn-equality}
\end{equation}
\end{proposition}

\noindent\emph{Proof sketch.}
For each coefficient let \(\mathcal G_j:H^{-1}(\Omega_1)\to
H^1_0(\Omega_1)\) be its variational Green operator.  Runge approximation and
the outer DN identity first give
\begin{equation}
  H(\mathcal G_1F)=H(\mathcal G_2F)
\end{equation}
for compact known-side sources \(F,H\).  Self-adjointness and an annihilator
argument then show that \(\mathcal G_1F\) and \(\mathcal G_2F\) have the same
quotient trace on \(\Sigma_2\) for every bounded interface functional \(F\).

For a global representative \(z\in H^1_0(\Omega_1)\), solve the Dirichlet
problem on each side of \(\Sigma_2\).  The difference of the two oriented
conormal functionals is bounded, depends only on the quotient trace of \(z\),
and is a left inverse to \(\mathcal G_j\) on interface functionals.  Gluing
shows that it is injective modulo equality of internal traces.  These facts give
equality of the conormal relations for the two coefficients.  Finally, each
smooth function compactly supported in \(\Gamma_2\) has its own global lift.
Testing the conormal identity on pairs of these lifts gives equality of the
inner DN forms, which extends by energy bounds to
\(H^{1/2}_0(\Gamma_2;\partial\Omega_2)\).  Appendix
\ref{app:inner-extension-proof} supplies the details.

\section{Proof of the main theorem}
\label{sec:global-uniqueness}

We iterate the one-cell theorem along an admissible recovery order, using
inner extension whenever a required face lies inside \(\Omega\).

\begin{proof}[Proof of Theorem \ref{thm:main}]
Let \(P_r^{(j)}\) denote the polynomial defining \(\sigma^{(j)}\) on
\(D_{\alpha_r}\), and let \(H_{r,i}\) support the accessible patch
\(\Gamma_{r,i}\).  We prove by induction on \(r\) that
\(P_r^{(1)}=P_r^{(2)}\).

The induction step reduces to equality of the Schur forms on the accessible
faces.  The given DN map supplies this equality on exterior faces; Proposition
\ref{prop:inner-extension} supplies it on internal ones.

For \(r=1\), every \(\Gamma_{1,i}\) lies in \(\Sigma\).  Restricting
\eqref{eq:outer-dn-equality} to this patch and applying Lemma
\ref{lem:face-data-from-local-dn} from the \(D_{\alpha_1}\)-side gives
\begin{equation}
  S_{H_{1,i}}(P_1^{(1)})
  =
  S_{H_{1,i}}(P_1^{(2)})
  \qquad\text{on }\Gamma_{1,i}.
  \label{eq:first-step-face-data}
\end{equation}
The hypotheses of Theorem \ref{thm:flat-cell-recovery} are part of Definition
\ref{def:stripping-order} and Theorem \ref{thm:main}; hence
\begin{equation}
  P_1^{(1)}=P_1^{(2)},
\end{equation}
and hence \(\sigma^{(1)}=\sigma^{(2)}\) in \(D_{\alpha_1}\).

Let \(r>1\), and assume that the conductivities agree on
\(D_{\alpha_1},\ldots,D_{\alpha_{r-1}}\).  Their common value is then known in
\(G_r=\Omega\setminus\overline{\Omega_r}\).  Fix an accessible patch
\(\Gamma_{r,i}\) of \(D_{\alpha_r}\).

If \(\Gamma_{r,i}\subset\Sigma\), restrict
\eqref{eq:outer-dn-equality} to \(\Gamma_{r,i}\) and apply Lemma
\ref{lem:face-data-from-local-dn} from the \(D_{\alpha_r}\)-side.  We obtain
\begin{equation}
  S_{H_{r,i}}(P_r^{(1)})
  =
  S_{H_{r,i}}(P_r^{(2)})
  \qquad\text{on }\Gamma_{r,i}.
  \label{eq:exterior-face-data}
\end{equation}
Now suppose that
\(\Gamma_{r,i}\subset\Xi_r=\Omega\cap\partial G_r\cap\partial\Omega_r\).
Extend the known coefficient on \(G_r\) to a uniformly elliptic piecewise
polynomial coefficient \(\sigma^0_r\) on \(\Omega\), so that
\begin{equation}
  \sigma^0_r=\sigma^{(1)}=\sigma^{(2)}
  \quad\text{in }G_r .
\end{equation}
For example, keep the recovered polynomials in \(G_r\) and choose any fixed
uniformly elliptic polynomial on each remaining cell.  The face-connectedness
of \(G_r\) to \(\Sigma_r\) and the exterior half-ball assumption allow us to
apply Lemma \ref{lem:runge-known-shell}, and hence Proposition
\ref{prop:inner-extension}, with
\begin{align}
  \Omega_1&=\Omega,&
  \Omega_2&=\Omega_r,&
  \Gamma_1&=\Sigma_r\subset\Sigma,\notag\\
  \Sigma_2&=\Xi_r,&
  \sigma^0&=\sigma^0_r,&
  \Gamma_2&=\Gamma_{r,i} .
\end{align}
The restriction of \eqref{eq:outer-dn-equality} to \(\Sigma_r\) is the outer
DN equality required in that proposition.  It follows directly that the DN
maps agree on the selected patch \(\Gamma_{r,i}\).  Applying Lemma
\ref{lem:face-data-from-local-dn} in \(\Omega_r\) gives
\begin{equation}
  S_{H_{r,i}}(P_r^{(1)})
  =
  S_{H_{r,i}}(P_r^{(2)})
  \qquad\text{on }\Gamma_{r,i}.
  \label{eq:interior-face-data}
\end{equation}
We have therefore obtained the required Schur equality on every accessible
face of \(D_{\alpha_r}\).  Theorem \ref{thm:flat-cell-recovery} gives
\begin{equation}
  P_r^{(1)}=P_r^{(2)},
\end{equation}
and hence \(\sigma^{(1)}=\sigma^{(2)}\) in \(D_{\alpha_r}\), completing the
induction.
\end{proof}

\begin{corollary}
\label{cor:uniform-degree}
Suppose \(N_\alpha\le N\) for all cells.  If the subdivision admits an
admissible flat-face recovery order relative to \(\Sigma\) in which each cell is
accessible through at least \(N+3\) flat faces whose supporting hyperplanes
satisfy the corresponding genericity and triple-intersection geometry, and the
recovered regions satisfy the stated face-connectedness assumptions, then the
following holds.  For any
\(\sigma^{(1)},\sigma^{(2)}\in\mathcal P(\mathcal D,N)\) that satisfy the
coefficient-dependent triple-region positivity and auxiliary half-ball
ellipticity hypotheses of Theorem \ref{thm:main} for this recovery order,
equality of their local DN maps on \(\Sigma\) implies
\(\sigma^{(1)}=\sigma^{(2)}\).  Equivalently, the local DN map is injective on
the subclass of \(\mathcal P(\mathcal D,N)\) satisfying those
coefficient-dependent admissibility conditions.
\end{corollary}

\appendix

\section{Details of the inner-extension argument}
\label{app:inner-extension-proof}

We prove Proposition \ref{prop:inner-extension} using only global variational
solutions and the ordinary trace space on the selected patch.  All closures and
neighborhoods involving \(\Omega_2\) are relative to \(\Omega_1\).

\subsection{Global traces and variational Green solutions}

Let \(\Omega_2\subset\Omega_1\subset\R^n\) be bounded Lipschitz domains, set
\begin{equation}
  G=\Omega_1\setminus\overline{\Omega_2},
  \qquad
  \Sigma_2=\Omega_1\cap\partial\Omega_2\cap\partial G,
\end{equation}
and put \(X=H^1_0(\Omega_1)\).  The restriction of an element of \(X\) to
\(\Omega_2\) has the standard quotient trace on \(\partial\Omega_2\); see
\cite[Chapters 3--4]{McLean2000}.  Write
\begin{equation}
  z\sim_{\Sigma_2}w
  \quad\Longleftrightarrow\quad
  \operatorname{Tr}_{\partial\Omega_2}(z-w)=0.
  \label{eq:internal-trace-equivalence}
\end{equation}
Because global elements of \(X\) have zero trace on \(\partial\Omega_1\), the
only variable part of this trace is the clipped interface \(\Sigma_2\).  Thus
\eqref{eq:internal-trace-equivalence} is precisely equality of the internal
trace, expressed as a relation between actual global representatives.  Let
\begin{equation}
  N=\{z\in X:z\sim_{\Sigma_2}0\},
  \qquad
  \mathcal T=N^\perp\subset X^*.
  \label{eq:interface-annihilator}
\end{equation}
The elements of \(\mathcal T\) are the bounded interface functionals.  This is
the usual identification \((X/N)^*=N^\perp\), obtained from Hahn--Banach; it
does not assert density of smooth functions in a dual space on the full
interface.

Let \(\sigma^{(j)}\), \(j=0,1,2\), denote the reference and the two unknown
coefficients, with \(\sigma^{(1)}=\sigma^{(2)}=\sigma^0\) in \(G\).  For the
reference coefficient in \(\Omega_2\) we use any fixed uniformly elliptic
extension.

\begin{lemma}
\label{lem:variational-green-sources}
For \(j=0,1,2\) and every \(F\in X^*=H^{-1}(\Omega_1)\), there is a unique
\(\mathcal G_jF\in X\) such that
\begin{equation}
  \int_{\Omega_1}\sigma^{(j)}\nabla\mathcal G_jF\cdot\nabla\phi\,dx
  =F(\phi),
  \qquad \phi\in X.
  \label{eq:variational-green}
\end{equation}
If \(\lambda_j\) is an ellipticity lower bound, then
\begin{equation}
  \|\mathcal G_jF\|_{H^1_0(\Omega_1)}
  \le C\lambda_j^{-1}\|F\|_{H^{-1}(\Omega_1)}.
  \label{eq:green-coercive-estimate}
\end{equation}
Moreover, each \(\mathcal G_j\) is symmetric:
\begin{equation}
  H(\mathcal G_jF)=F(\mathcal G_jH),
  \qquad F,H\in X^*.
  \label{eq:green-symmetry}
\end{equation}
\end{lemma}

\begin{proof}
Poincar\'e's inequality and uniform ellipticity make the bilinear form in
\eqref{eq:variational-green} bounded and coercive on \(X\).  Lax--Milgram gives
the solution and \eqref{eq:green-coercive-estimate}; see, for example,
\cite[Chapter 6]{Evans2010}.  Testing the equations for \(F\) and \(H\) against
the opposite solutions proves \eqref{eq:green-symmetry}.
\end{proof}

A \emph{compact known-side source} will mean either a functional represented
by a density in \(C_c^\infty(G)\), or a smooth density on a flat surface
parallel to one interface face whose support is compactly contained in \(G\).
In the second case the surface and its support are chosen at positive distance
from \(\Omega_2\).  Both types belong to \(X^*\) by the local trace theorem.  If
\(p\in C_c^\infty(\Gamma_2)\) and flat coordinates put its face at \(x_n=0\),
the offset functionals
\begin{equation}
  F_{p,\varepsilon}(\phi)
  =\int p(x')\phi(x',\varepsilon)\,dx'
  \label{eq:offset-source}
\end{equation}
are compact known-side sources for sufficiently small \(\varepsilon>0\).  The
one-dimensional fundamental theorem of calculus gives
\begin{equation}
  \|F_{p,\varepsilon}-F_p\|_{X^*}
  \le C\varepsilon^{1/2}\|p\|_{L^2},
  \qquad
  F_p(\phi)=\int_{\Gamma_2}p\,\operatorname{Tr}\phi\,dS .
  \label{eq:offset-source-limit}
\end{equation}
This approximation is asserted only for an individual smooth face density; no
norm-density claim for the full interface dual is used.

\subsection{Unique continuation and Runge approximation}

A union of cells is \emph{face-connected} if any two of its cells can be joined
by a chain in which consecutive cells share a nonempty relatively open flat
face.  The following unique continuation result is the scalar conductivity
analogue of the Holmgren argument in \cite{CarsteaHondaNakamura2018}.

\begin{lemma}
\label{lem:piecewise-poly-ucp}
Let \(G\) be a face-connected finite union of Lipschitz subdomains whose
pairwise interfaces contain relatively open flat faces.  Assume that on each
subdomain \(E\subset G\), the coefficient \(\sigma^0\) is the restriction of a
symmetric matrix polynomial which remains uniformly elliptic in a neighborhood
of \(\overline E\).  The union may include one auxiliary analytic collar or
ball attached to a flat exterior face, with \(\sigma^0\) given there by the same
polynomial extension as in the adjacent cell of the known region.  If
\(u\in H^1_{\mathrm{loc}}(G)\) satisfies
\begin{equation}
  \diver(\sigma^0\nabla u)=0\quad\text{in }G
\end{equation}
and \(u\) vanishes in a nonempty open subset of \(G\), then \(u\equiv0\) in
\(G\).
\end{lemma}

\begin{proof}
On each subdomain the coefficient is real analytic and uniformly elliptic.
Analytic elliptic regularity makes the solution real analytic in the interior,
and analytic uniqueness then propagates vanishing from a nonempty open set
throughout that subdomain; see
\cite[Theorem 8.6.5, p.~309]{Hormander1983} (with the analytic-ellipticity
conventions on pp.~271 and 306).

It remains to cross a flat interface.  Let neighboring
subdomains \(E_-\) and \(E_+\) share a relatively open flat face \(F\), and
suppose that \(u\) vanishes in \(E_-\) up to a nonempty subpatch
\(F_0\Subset F\).  Since \(u\in H^1_{\mathrm{loc}}\) is a weak solution across
\(F\), the two traces of \(u\) agree in \(H^{1/2}_{\mathrm{loc}}(F)\), and the
conormal fluxes agree in \(H^{-1/2}_{\mathrm{loc}}(F)\): for every test function
supported near \(F_0\), integration by parts on the two sides gives
\begin{equation}
  [u]_{F_0}=0,
  \qquad
  [\sigma^0\nabla u\cdot\nu]_{F_0}=0 .
  \label{eq:transmission-traces-ucp}
\end{equation}
The trace and conormal derivative vanish on the \(E_-\)-side because \(u=0\)
there.  They therefore vanish on the \(E_+\)-side as well.  The tangential
derivatives of the zero trace vanish, and ellipticity recovers the normal
derivative from the conormal derivative.  Thus \(u|_{E_+}\) has zero Cauchy
data on the noncharacteristic analytic hypersurface \(F_0\).

Extend the polynomial coefficient from \(E_+\) across \(F_0\), and extend
\(u|_{E_+}\) by zero.  The zero Dirichlet trace makes this an \(H^1\) extension,
and the zero conormal trace removes the interface term from the weak equation.
The extension is therefore a weak solution of an analytic uniformly elliptic
equation in a full neighborhood of \(F_0\).  Analytic elliptic regularity
applies before any use of Cauchy uniqueness and makes the extension real
analytic there.

Holmgren's theorem \cite{John1982} gives \(u=0\) on the \(E_+\)-side near
\(F_0\), and analytic continuation gives \(u=0\) throughout \(E_+\).  Iteration
along a face chain proves the lemma.
\end{proof}

\begin{lemma}
\label{lem:runge-known-shell}
Assume that \(G=\Omega_1\setminus\overline{\Omega_2}\) is a face-connected
finite union of cells as in Lemma \ref{lem:piecewise-poly-ucp}, and that
\(\Gamma_1\subset\Sigma_1\) contains a relatively compact nonempty flat
subpatch in the following explicit sense: there is a ball \(B\) such that,
after an affine change of coordinates in \(B\),
\begin{equation}
  \Omega_1\cap B=\{x_n>0\}\cap B,
  \qquad
  B\cap\partial\Omega_1=\{x_n=0\}\cap B\Subset \Gamma_1,
\end{equation}
and the coefficient in the cell of the known region adjacent to \(B\cap\partial\Omega_1\)
extends as a uniformly elliptic polynomial coefficient to the whole ball
\(B\).  Let \(\sigma^0\) be a known
uniformly elliptic piecewise polynomial coefficient on \(\Omega_1\), agreeing
with the already recovered coefficient in \(G\), and satisfying the hypotheses
of Lemma \ref{lem:piecewise-poly-ucp} in \(G\).  If
\(V\subset\Omega_1\) is relatively open and
\begin{equation}
  \overline{\Omega_2}^{\,\Omega_1}:=\overline{\Omega_2}\cap\Omega_1
  \subset V,
\end{equation}
and if \(u\in H^1_0(\Omega_1)\) satisfies
\begin{equation}
  \diver(\sigma^0\nabla u)=0\quad\text{in }V,
\end{equation}
where the equation in \(V\) is understood in the relative sense inside
\(\Omega_1\), then there are functions \(u_k\in H^1(\Omega_1)\), solving
\begin{equation}
  \diver(\sigma^0\nabla u_k)=0\quad\text{in }\Omega_1,
\end{equation}
whose boundary traces are supported in \(\Gamma_1\), such that
\begin{equation}
  u_k\to u\quad\text{in }H^1(\Omega_2).
\end{equation}
\end{lemma}

\begin{proof}
Shrink \(B\), if necessary, without losing the flatness and extension
properties, and set
\begin{equation}
  \Omega_0=\Omega_1\cup B .
\end{equation}
In the chosen coordinates, \(\Omega_0\) is obtained from \(\Omega_1\) by
attaching the exterior half-ball
\(B^-:=B\cap\{x_n<0\}\) across the flat patch
\(B\cap\partial\Omega_1\).  Extend \(\sigma^0\) to \(B\) by the polynomial from
the adjacent cell.  The half-ball is then an auxiliary analytic cell of the
kind allowed in Lemma \ref{lem:piecewise-poly-ucp}.

Let \(\mathcal G_{\Omega_0}\) denote the Green operator for the extended
coefficient \(\sigma^0\) on \(\Omega_0\) with zero Dirichlet boundary
condition.  Consider the subspaces of \(H^1(\Omega_2)\)
\begin{align}
  X&=\{v|_{\Omega_2}:\ v\in H^1_0(\Omega_1),\
        \diver(\sigma^0\nabla v)=0\text{ in }V\},\\
  Y&=\{(\mathcal G_{\Omega_0}F)|_{\Omega_2}:\ 
        F\in H^{-1}(\Omega_0),\
        \supp F\Subset B^-\}.
\end{align}
Every element of \(Y\) extends to a solution in \(\Omega_1\) whose boundary
trace is supported in \(B\cap\partial\Omega_1\Subset\Gamma_1\).  It therefore
suffices to show that \(Y\) is dense in \(X\).

Let \(\ell\in (H^1(\Omega_2))^*\) annihilate \(Y\), and extend it to
\(\widetilde\ell\in H^{-1}(\Omega_0)\) by
\begin{equation}
  \widetilde\ell(\phi)=\ell(\phi|_{\Omega_2}),
  \qquad \phi\in H^1_0(\Omega_0).
\end{equation}
Let \(w=\mathcal G_{\Omega_0}\widetilde\ell\).  For every
\(F\in H^{-1}(\Omega_0)\) with \(\supp F\Subset B^-\), self-adjointness gives
\begin{equation}
  0=\ell((\mathcal G_{\Omega_0}F)|_{\Omega_2})
   =\widetilde\ell(\mathcal G_{\Omega_0}F)
   =F(w).
\end{equation}
Thus \(w=0\) in \(B^-\).  The distribution \(\widetilde\ell\) is supported in
\(\overline{\Omega_2}^{\,\Omega_1}\), so \(w\) is homogeneous in \(G\cup B^-\).
Lemma \ref{lem:piecewise-poly-ucp} applied to this face-connected union gives
\begin{equation}
  w=0\quad\text{in }G
\end{equation}
as well as in the auxiliary half-ball.

Take \(v|_{\Omega_2}\in X\) and extend \(v\) by zero from \(\Omega_1\) to
\(\Omega_0\).  Because \(w\in H^1_0(\Omega_0)\) vanishes in \(G\), its
restriction to \(\Omega_2\) has zero trace on \(\Sigma_2\) and on the part of
\(\partial\Omega_2\) lying in \(\partial\Omega_1\).  Extending this restriction
by zero to \(\Omega_1\) gives a valid test function for the equation satisfied
by \(v\) in \(V\).  Therefore
\begin{equation}
  \ell(v|_{\Omega_2})
  =\widetilde\ell(v)
  =\int_{\Omega_0}\sigma^0\nabla w\cdot\nabla v\,dx
  =\int_{\Omega_2}\sigma^0\nabla w\cdot\nabla v\,dx
  =0.
\end{equation}
Thus every functional that annihilates \(Y\) also annihilates \(X\).
Hahn--Banach proves the required density.
\end{proof}

\subsection{Compact Green pairings and internal conormal relations}

\begin{lemma}
\label{lem:compact-green-pairings}
Under the hypotheses of Proposition \ref{prop:inner-extension}, equality of the
outer local DN maps gives
\begin{equation}
  H(\mathcal G_1F)=H(\mathcal G_2F)
  \label{eq:compact-green-pairings}
\end{equation}
for every pair of compact sources \(F,H\) in the known region.
\end{lemma}

\begin{proof}
Fix compact known-side sources \(F,H\), and set
\(U_0=\mathcal G_0F\) and \(V_0=\mathcal G_0H\).  Because the sources are
supported away from \(\Omega_2\), both functions solve the reference
homogeneous equation in a relative neighborhood of
\(\overline{\Omega_2}^{\,\Omega_1}\).  Lemma \ref{lem:runge-known-shell} gives
reference solutions \(U_k,V_k\) in
\(\Omega_1\), with boundary traces supported in \(\Gamma_1\), such that
\begin{equation}
  U_k\to U_0,
  \qquad
  V_k\to V_0
  \quad\text{in }H^1(\Omega_2).
  \label{eq:runge-approximation-pair}
\end{equation}
For \(j=1,2\), let \(Z_{j,k}\) solve
\(\diver(\sigma^{(j)}\nabla Z_{j,k})=0\) in \(\Omega_1\) with
\(Z_{j,k}|_{\partial\Omega_1}=U_k|_{\partial\Omega_1}\), and set
\(W_{j,k}=Z_{j,k}-U_k\in H^1_0(\Omega_1)\).  Since
\(\sigma^{(j)}-\sigma^0\) is supported in \(\Omega_2\), the function
\(W_{j,k}\) is the unique solution of
\begin{equation}
  \int_{\Omega_1}\sigma^{(j)}\nabla W_{j,k}\cdot\nabla\phi\,dx
  =-\int_{\Omega_2}(\sigma^{(j)}-\sigma^0)
       \nabla U_k\cdot\nabla\phi\,dx,
  \qquad \phi\in H^1_0(\Omega_1).
  \label{eq:Wjk-equation}
\end{equation}
The right-hand side of \eqref{eq:Wjk-equation} converges in
\(H^{-1}(\Omega_1)\), since
\begin{equation}
  \left|
  \int_{\Omega_2}(\sigma^{(j)}-\sigma^0)
       \nabla (U_k-U_0)\cdot\nabla\phi\,dx
  \right|
  \le C\|U_k-U_0\|_{H^1(\Omega_2)}\|\phi\|_{H^1_0(\Omega_1)} .
\end{equation}
The limit is the equation for \(W_j=\mathcal G_jF-\mathcal G_0F\).  The two
Green solutions have the same source, and their coefficients agree in \(G\),
so
\begin{equation}
  \int_{\Omega_1}\sigma^{(j)}\nabla W_j\cdot\nabla\phi\,dx
  =-\int_{\Omega_2}(\sigma^{(j)}-\sigma^0)
       \nabla U_0\cdot\nabla\phi\,dx .
\end{equation}
The energy estimate therefore gives
\begin{equation}
  \|W_{j,k}-W_j\|_{H^1_0(\Omega_1)}
  \le C\|U_k-U_0\|_{H^1(\Omega_2)}\to0.
  \label{eq:Wjk-convergence}
\end{equation}

The local DN identity, applied to the reference solutions \(U_k,V_k\), gives
\begin{align}
 &\langle
 (\DN^{\Omega_1}_{\sigma^{(j)},\Gamma_1}
  -\DN^{\Omega_1}_{\sigma^0,\Gamma_1})
      U_k|_{\partial\Omega_1},V_k|_{\partial\Omega_1}\rangle
      \notag\\
 &\quad =
  \int_{\Omega_2}(\sigma^{(j)}-\sigma^0)
    \nabla (U_k+W_{j,k})\cdot\nabla V_k\,dx .
  \label{eq:dn-difference-pairing}
\end{align}
The convergences in \eqref{eq:runge-approximation-pair} and
\eqref{eq:Wjk-convergence} allow us to pass to the limit:
\begin{equation}
  \int_{\Omega_2}(\sigma^{(j)}-\sigma^0)
    \nabla (U_k+W_{j,k})\cdot\nabla V_k\,dx
  \longrightarrow
  \int_{\Omega_2}(\sigma^{(j)}-\sigma^0)
    \nabla (U_0+W_j)\cdot\nabla V_0\,dx .
  \label{eq:dn-pairing-limit}
\end{equation}
Because \(V_0=\mathcal G_0H\),
\(H(W_j)=\int_{\Omega_1}\sigma^0\nabla V_0\cdot\nabla W_j\,dx\).  Testing the
limit equation for \(W_j\) with \(V_0\) yields
\begin{equation}
  \int_{\Omega_1}\sigma^{(j)}\nabla W_j\cdot\nabla V_0\,dx
  =-\int_{\Omega_2}(\sigma^{(j)}-\sigma^0)
    \nabla U_0\cdot\nabla V_0\,dx .
\end{equation}
Using \(\sigma^{(j)}=\sigma^0\) in \(G\), we obtain
\begin{equation}
  H(W_j)
  =-\int_{\Omega_2}(\sigma^{(j)}-\sigma^0)
    \nabla (U_0+W_j)\cdot\nabla V_0\,dx .
\end{equation}
Together with \eqref{eq:dn-pairing-limit}, this gives
\begin{equation}
  H(W_j)
  =-\lim_{k\to\infty}
  \int_{\Omega_2}(\sigma^{(j)}-\sigma^0)
    \nabla (U_k+W_{j,k})\cdot\nabla V_k\,dx .
  \label{eq:green-pairing-from-dn}
\end{equation}
Hence the outer local DN map and the reference coefficient determine the
pairing \(H(\mathcal G_jF)\).  The two outer maps agree by
\eqref{eq:outer-local-equality-inner-section}; therefore
\eqref{eq:compact-green-pairings} holds.  The convergence in
\eqref{eq:Wjk-convergence} and the estimate
\eqref{eq:green-coercive-estimate} justify every passage to the Runge limit.
\end{proof}

\begin{lemma}
\label{lem:green-trace-relation}
For every bounded interface functional \(F\in\mathcal T\),
\begin{equation}
  \mathcal G_1F\sim_{\Sigma_2}\mathcal G_2F.
  \label{eq:green-traces-equal}
\end{equation}
\end{lemma}

\begin{proof}
We first take a compact known-side source \(F\), and put
\(w=\mathcal G_1F-\mathcal G_2F\in X\).  Lemma
\ref{lem:compact-green-pairings} gives \(H(w)=0\) for every compact known-side
source \(H\).  In particular this holds for every smooth bulk density compactly
supported in \(G\), so \(w=0\) in \(G\) as a distribution and hence in
\(H^1(G)\).  The trace theorem and the fact that \(w\) is one global \(H^1\)
function show that its quotient trace from \(\Omega_2\) vanishes on
\(\Sigma_2\).  Thus
\begin{equation}
  \mathcal G_1F\sim_{\Sigma_2}\mathcal G_2F
  \qquad\text{for compact known-side }F.
  \label{eq:compact-source-trace-equality}
\end{equation}

Now let \(F\in\mathcal T=N^\perp\) be arbitrary.  For every compact known-side
source \(H\), Green symmetry and \eqref{eq:compact-source-trace-equality} give
\begin{align}
 H(\mathcal G_1F-\mathcal G_2F)
 &=F(\mathcal G_1H-\mathcal G_2H)=0,
\end{align}
because the difference belongs to \(N\).  Testing again with all smooth bulk
densities compactly supported in \(G\) shows that
\(\mathcal G_1F-\mathcal G_2F=0\) in \(G\), and trace continuity gives
\eqref{eq:green-traces-equal}.  The first passage takes place in the
\(X\)--\(X^*\) duality, the vanishing is distributional in \(G\), and the last
passage uses the bounded trace map from \(H^1(G)\).  The quotient-dual
identification \(\mathcal T=N^\perp\) is the Hahn--Banach step; no approximation
of \(F\) by smooth full-interface densities has been made.
\end{proof}

For the next lemma, let \(u_j[z]\in H^1(\Omega_2)\) be the
\(\sigma^{(j)}\)-harmonic function with boundary trace equal to that of
\(z|_{\Omega_2}\), and let \(v^0[z]\in H^1(G)\) be the
\(\sigma^0\)-harmonic function with the boundary trace of \(z|_G\).  In
particular, the latter trace is zero on the outer part of \(\partial G\).
Existence, uniqueness, and the energy bounds are the standard Lipschitz-domain
Dirichlet theory; see \cite[Chapters 3--4]{McLean2000}.

\begin{lemma}
\label{lem:conormal-difference-functionals}
For \(z\in X\), define \(K_jz\in X^*\) by
\begin{equation}
\begin{split}
  (K_jz)(\phi)
  &=\int_{\Omega_2}\sigma^{(j)}\nabla u_j[z]\cdot\nabla\phi\,dx
    -\int_G^{(+)}\sigma^0\nabla v^0[z]\cdot\nabla\phi\,dx\\
  &=\int_{\Omega_2}\sigma^{(j)}\nabla u_j[z]\cdot\nabla\phi\,dx
    +\int_G\sigma^0\nabla v^0[z]\cdot\nabla\phi\,dx,
  \qquad \phi\in X.
  \label{eq:conormal-difference-functional}
\end{split}
\end{equation}
Here \(\int_G^{(+)}\) denotes the known-side Green form oriented with the
normal pointing out of \(\Omega_2\); it is the negative of the usual volume
energy on \(G\).  Thus the first line is ``inner conormal minus known-side
conormal'', while the second line is the corresponding volume formula.

The map \(K_j:X\to\mathcal T\) is bounded and depends only on the internal
trace class of \(z\).  It satisfies
\begin{equation}
  K_j\mathcal G_jF=F,
  \qquad F\in\mathcal T,
  \label{eq:conormal-left-inverse}
\end{equation}
as an equality in \(X^*\), and it is injective modulo internal trace:
\begin{equation}
  K_jz=K_jw\quad\Longrightarrow\quad z\sim_{\Sigma_2}w.
  \label{eq:conormal-trace-injectivity}
\end{equation}
\end{lemma}

\begin{proof}
The two Dirichlet energy estimates and the trace theorem give
\begin{equation}
  \|K_jz\|_{X^*}\le C\|z\|_X.
\end{equation}
Both Dirichlet solutions depend only on the
trace of \(z\), so the same is true of \(K_jz\).  If \(\phi\in N\), its trace
vanishes on the internal interface and on the outer boundary; using its
restrictions as zero-boundary test functions in the two homogeneous equations
shows \((K_jz)(\phi)=0\).  Hence \(K_jz\in N^\perp=\mathcal T\).

Let \(F\in\mathcal T\) and \(U=\mathcal G_jF\).  Since \(F\) annihilates
\(N\), testing with functions compactly supported on either side shows that
\(U\) is homogeneous in \(\Omega_2\) and in \(G\).  Its restrictions are
therefore exactly \(u_j[U]\) and \(v^0[U]\).  The volume form in
\eqref{eq:conormal-difference-functional} and
\eqref{eq:variational-green} now give
\begin{equation}
  (K_jU)(\phi)
  =\int_{\Omega_1}\sigma^{(j)}\nabla U\cdot\nabla\phi\,dx
  =F(\phi),
\end{equation}
which proves \eqref{eq:conormal-left-inverse}.

Finally suppose \(K_jz=0\).  Glue \(u_j[z]\) and \(v^0[z]\).  Their traces
match because both come from the global representative \(z\), so the glued
function belongs to \(X\).  Equation
\eqref{eq:conormal-difference-functional} says precisely that its weak global
conductivity form vanishes.  Uniqueness for the zero-boundary problem makes the
glued function zero, and hence \(z\sim_{\Sigma_2}0\).  Applying this to \(z-w\)
proves \eqref{eq:conormal-trace-injectivity}.
\end{proof}

\begin{lemma}
\label{lem:conormal-relations-equal}
For every \(z\in X\),
\begin{equation}
  K_1z=K_2z\quad\text{in }X^*.
  \label{eq:conormal-relations-equal}
\end{equation}
\end{lemma}

\begin{proof}
Fix \(z\in X\) and put \(F=K_1z\in\mathcal T\).  The left-inverse identity
gives \(K_1\mathcal G_1F=F=K_1z\); injectivity modulo trace therefore yields
\(\mathcal G_1F\sim_{\Sigma_2}z\).  Lemma \ref{lem:green-trace-relation} gives
\(\mathcal G_2F\sim_{\Sigma_2}\mathcal G_1F\), and hence
\(\mathcal G_2F\sim_{\Sigma_2}z\).  Since \(K_2\) depends only on the trace
class and \(K_2\mathcal G_2F=F\), we conclude that
\(K_2z=F=K_1z\).
\end{proof}

\begin{lemma}
\label{lem:selected-patch-from-conormal}
If \eqref{eq:conormal-relations-equal} holds, then
\eqref{eq:inner-dn-equality} holds on the selected flat patch \(\Gamma_2\).
\end{lemma}

\begin{proof}
Take \(p,q\in C_c^\infty(\Gamma_2)\).  Because each support is compactly
contained in one flat face and avoids the skeleton, a cutoff in a graph chart
gives individual lifts \(z_p,z_q\in X\) whose traces on
\(\partial\Omega_2\) equal \(p,q\), respectively, and vanish off the selected
patch.  No simultaneous extension of an arbitrary full-interface trace is
needed.

Apply \eqref{eq:conormal-relations-equal} to \(z_p\) and test the resulting
global functionals on \(z_q\).  The known-side terms in
\eqref{eq:conormal-difference-functional} cancel because their coefficient and
Dirichlet solutions are fixed.  The remaining identity is
\begin{equation}
  \bigl\langle\DN^{\Omega_2}_{\sigma^{(1)},\Gamma_2}p,q\bigr\rangle
  =
  \bigl\langle\DN^{\Omega_2}_{\sigma^{(2)},\Gamma_2}p,q\bigr\rangle .
  \label{eq:smooth-selected-patch-dn}
\end{equation}
The Dirichlet energy estimate bounds either side by
\begin{equation}
  C\|p\|_{H^{1/2}(\partial\Omega_2)}
    \|q\|_{H^{1/2}(\partial\Omega_2)}.
\end{equation}
First approximate the input and then the test trace by smooth functions compactly
supported in \(\Gamma_2\).  The defining completion of
\(H^{1/2}_0(\Gamma_2;\partial\Omega_2)\) extends
\eqref{eq:smooth-selected-patch-dn} to the whole selected-patch space.
\end{proof}

\begin{proof}[Proof of Proposition \ref{prop:inner-extension}]
Lemma \ref{lem:compact-green-pairings} follows from the outer DN identity and
Runge approximation.  Lemmas \ref{lem:green-trace-relation}--
\ref{lem:conormal-relations-equal} then identify the two conormal relations on
the quotient traces of global \(H^1_0(\Omega_1)\) representatives.  Lemma
\ref{lem:selected-patch-from-conormal} gives the asserted equality of the local
DN maps on \(\Gamma_2\).
\end{proof}

\section*{Acknowledgements}
The author was supported by NSTC grant 113-2115-M-A49-018-MY3.

\bibliographystyle{amsplain}
\bibliography{paper_refs}

@misc{AlessandriniDeHoopGaburroSincich2016,
  author = {Alessandrini, G. and de Hoop, M. V. and Gaburro, R. and Sincich, E.},
  title  = {Lipschitz stability for the electrostatic inverse boundary value problem with piecewise linear conductivities},
  year   = {2016},
  note   = {Journal de Math{\'e}matiques Pures et Appliqu{\'e}es, doi:10.1016/j.matpur.2016.10.001},
}

@article{AlessandriniVessella2005,
  author  = {Alessandrini, G. and Vessella, S.},
  title   = {Lipschitz stability for the inverse conductivity problem},
  journal = {Advances in Applied Mathematics},
  volume  = {35},
  year    = {2005},
  number  = {2},
  pages   = {207--241},
}

@misc{BerettaDeHoopQiu2012,
  author        = {Beretta, E. and de Hoop, M. V. and Qiu, L.},
  title         = {Lipschitz stability of an inverse boundary value problem for a {Schr{\"o}dinger} type equation},
  year          = {2012},
  eprint        = {1203.1650},
  archivePrefix = {arXiv},
  primaryClass  = {math.AP},
  note          = {arXiv:1203.1650},
}

@misc{BerettaDeHoopQiuScherzer2014,
  author        = {Beretta, E. and de Hoop, M. V. and Qiu, L. and Scherzer, O.},
  title         = {Inverse boundary value problem for the {Helmholtz} equation: multi-level approach and iterative reconstruction},
  year          = {2014},
  eprint        = {1406.2391},
  archivePrefix = {arXiv},
  primaryClass  = {math.NA},
  note          = {arXiv:1406.2391},
}

@misc{BerettaFrancini2011,
  author = {Beretta, E. and Francini, E.},
  title  = {Lipschitz stability for the electrical impedance tomography problem: the complex case},
  year   = {2011},
  note   = {Communications in Partial Differential Equations, 36(10), doi:10.1080/03605302.2011.552930; arXiv:1008.4046},
}

@misc{BerettaFranciniVessella2013,
  author        = {Beretta, E. and Francini, E. and Vessella, S.},
  title         = {Uniqueness and {Lipschitz} stability for the identification of {Lam{\'e}} parameters from boundary measurements},
  year          = {2013},
  eprint        = {1303.2443},
  archivePrefix = {arXiv},
  primaryClass  = {math.AP},
  note          = {arXiv:1303.2443},
}

@misc{CarsteaNakamuraOksanen2019,
  author = {C{\^a}rstea, C. I. and Nakamura, G. and Oksanen, L.},
  title  = {Uniqueness for the inverse boundary value problem of piecewise homogeneous anisotropic elasticity in the time domain},
  year   = {2019},
  note   = {arXiv:1903.01178},
}

@article{GaburroSincich2015,
  author  = {Gaburro, R. and Sincich, E.},
  title   = {Lipschitz stability for the inverse conductivity problem for a conformal class of anisotropic conductivities},
  journal = {Inverse Problems},
  volume  = {31},
  year    = {2015},
  number  = {1},
  pages   = {015008},
}

@article{Garde2020,
  author  = {Garde, H.},
  title   = {Reconstruction of piecewise constant layered conductivities in electrical impedance tomography},
  journal = {Communications in Partial Differential Equations},
  volume  = {45},
  year    = {2020},
  number  = {9},
  pages   = {1118--1133},
  note    = {doi:10.1080/03605302.2020.1760884; arXiv:1904.07775},
}

@article{Isakov1988,
  author  = {Isakov, V.},
  title   = {On uniqueness of recovery of a discontinuous conductivity coefficient},
  journal = {Communications on Pure and Applied Mathematics},
  volume  = {41},
  year    = {1988},
  number  = {7},
  pages   = {865--877},
}

@article{KohnVogelius1984,
  author  = {Kohn, R. V. and Vogelius, M.},
  title   = {Determining conductivity by boundary measurements},
  journal = {Communications on Pure and Applied Mathematics},
  volume  = {37},
  year    = {1984},
  number  = {3},
  pages   = {289--298},
}

@article{KohnVogelius1985,
  author  = {Kohn, R. V. and Vogelius, M.},
  title   = {Determining conductivity by boundary measurements. {II}. Interior results},
  journal = {Communications on Pure and Applied Mathematics},
  volume  = {38},
  year    = {1985},
  number  = {5},
  pages   = {643--667},
}

@article{AlessandriniDeHoopGaburro2017,
  author  = {Alessandrini, G. and de Hoop, M. V. and Gaburro, R.},
  title   = {Uniqueness for the electrostatic inverse boundary value problem with piecewise constant anisotropic conductivities},
  journal = {Inverse Problems},
  volume  = {33},
  year    = {2017},
  number  = {12},
  pages   = {125013},
}

@article{AstalaPaivarinta2006,
  author  = {Astala, K. and P{\"a}iv{\"a}rinta, L.},
  title   = {Calderon's inverse conductivity problem in the plane},
  journal = {Annals of Mathematics (2)},
  volume  = {163},
  year    = {2006},
  number  = {1},
  pages   = {265--299},
}

@article{BrownUhlmann1997,
  author  = {Brown, R. M. and Uhlmann, G. A.},
  title   = {Uniqueness in the inverse conductivity problem for nonsmooth conductivities in two dimensions},
  journal = {Communications in Partial Differential Equations},
  volume  = {22},
  year    = {1997},
  number  = {5--6},
  pages   = {1009--1027},
}

@incollection{Calderon1980,
  author    = {Calder{\'o}n, A. P.},
  title     = {On an inverse boundary value problem},
  booktitle = {Seminar on Numerical Analysis and its Applications to Continuum Physics},
  address   = {Rio de Janeiro},
  publisher = {Sociedade Brasileira de Matem{\'a}tica},
  year      = {1980},
  pages     = {65--73},
}

@article{CarsteaHondaNakamura2018,
  author  = {C{\^a}rstea, C. I. and Honda, N. and Nakamura, G.},
  title   = {Uniqueness in the inverse boundary value problem for piecewise homogeneous anisotropic elasticity},
  journal = {SIAM Journal on Mathematical Analysis},
  volume  = {50},
  year    = {2018},
  number  = {3},
  pages   = {3291--3302},
}

@article{Ikehata2002,
  author  = {Ikehata, M.},
  title   = {Reconstruction of inclusion from boundary measurements},
  journal = {Journal of Inverse and Ill-Posed Problems},
  volume  = {10},
  year    = {2002},
  number  = {1},
  pages   = {37--65},
}

@book{John1982,
  author    = {John, F.},
  title     = {Partial Differential Equations},
  edition   = {Fourth},
  series    = {Applied Mathematical Sciences},
  volume    = {1},
  publisher = {Springer-Verlag},
  address   = {New York},
  year      = {1982},
}

@book{Evans2010,
  author    = {Evans, Lawrence C.},
  title     = {Partial Differential Equations},
  edition   = {Second},
  series    = {Graduate Studies in Mathematics},
  volume    = {19},
  publisher = {American Mathematical Society},
  address   = {Providence, RI},
  year      = {2010},
}

@book{Hormander1983,
  author    = {H{\"o}rmander, Lars},
  title     = {The Analysis of Linear Partial Differential Operators {I}: Distribution Theory and Fourier Analysis},
  series    = {Grundlehren der mathematischen Wissenschaften},
  volume    = {256},
  publisher = {Springer-Verlag},
  address   = {Berlin},
  year      = {1983},
}

@book{McLean2000,
  author    = {McLean, William},
  title     = {Strongly Elliptic Systems and Boundary Integral Equations},
  publisher = {Cambridge University Press},
  address   = {Cambridge},
  year      = {2000},
}

@article{KangYun2003,
  author  = {Kang, H. and Yun, K.},
  title   = {Boundary determination of conductivities and {Riemannian} metrics via local {Dirichlet-to-Neumann} operator},
  journal = {SIAM Journal on Mathematical Analysis},
  volume  = {34},
  year    = {2003},
  number  = {3},
  pages   = {719--735},
}

@article{LassasUhlmann2001,
  author  = {Lassas, M. and Uhlmann, G.},
  title   = {On determining a {Riemannian} manifold from the {Dirichlet-to-Neumann} map},
  journal = {Annales scientifiques de l'{\'E}cole Normale Sup{\'e}rieure (4)},
  volume  = {34},
  year    = {2001},
  number  = {5},
  pages   = {771--787},
}

@article{LeeUhlmann1989,
  author  = {Lee, J. M. and Uhlmann, G.},
  title   = {Determining anisotropic real-analytic conductivities by boundary measurements},
  journal = {Communications on Pure and Applied Mathematics},
  volume  = {42},
  year    = {1989},
  number  = {8},
  pages   = {1097--1112},
}

@article{Nachman1988,
  author  = {Nachman, A. I.},
  title   = {Reconstructions from boundary measurements},
  journal = {Annals of Mathematics (2)},
  volume  = {128},
  year    = {1988},
  number  = {3},
  pages   = {531--576},
}

@article{Sylvester1990,
  author  = {Sylvester, J.},
  title   = {An anisotropic inverse boundary value problem},
  journal = {Communications on Pure and Applied Mathematics},
  volume  = {43},
  year    = {1990},
  number  = {2},
  pages   = {201--232},
}

@article{SylvesterUhlmann1987,
  author  = {Sylvester, J. and Uhlmann, G.},
  title   = {A global uniqueness theorem for an inverse boundary value problem},
  journal = {Annals of Mathematics (2)},
  volume  = {125},
  year    = {1987},
  number  = {1},
  pages   = {153--169},
}

@article{Uhlmann2009,
  author  = {Uhlmann, G.},
  title   = {Electrical impedance tomography and {Calderon's} problem},
  journal = {Inverse Problems},
  volume  = {25},
  year    = {2009},
  number  = {12},
  pages   = {123011},
}

@misc{Carstea2026HolderStability,
  author = {C{\^a}rstea, C. I.},
  title  = {H{\"o}lder Stability from Exact Uniqueness for Finite-Dimensional Analytic Inverse Problems},
  year   = {2026},
  note   = {arXiv:2605.06354},
}

\end{document}